\documentclass[preprint,12pt]{elsarticle}

\usepackage{amssymb}
\usepackage{amsmath, amsthm}

\newtheorem{theorem}{Theorem}[section]
\theoremstyle{plain}

\newtheorem{corollary}[theorem]{Corollary}

\newtheorem{example}[theorem]{Example}

\newtheorem{lemma}[theorem]{Lemma}

\newtheorem{proposition}[theorem]{Proposition}

\numberwithin{equation}{section}

\journal{Linear Algebra and its Applications}

\def\gf#1{\mathbb F_{#1}}
\def\df#1#2{\Delta^{\hskip -1pt#1}#2}               
\def\chr#1{\mathrm{char}(#1)}                       
\def\combdeg#1{\mathrm{cdeg}(#1)}
\def\mvar#1{\overline{#1}}
\def\mmon#1#2{\overline{#1}^{\,\overline{#2}}}
\def\var#1{\mathrm{var}(#1)}
\def\weight{\omega}

\def\pl#1#2{\mathcal{P}_{#2}(#1)}                    
\def\tpl#1#2{\mathcal{P}^t_{#2}(#1)}             
\def\dpl#1#2{\mathcal{P}^\equiv_{#2}(#1)}            

\begin{document}

\begin{frontmatter}



\title{Symmetric multilinear forms and polarization of polynomials}


\author[D]{Ale\v{s} Dr\'apal\fnref{D2}}

\address[D]{Department of Algebra, Charles University, Sokolovsk\'a 183,
Praha 186 75, Czech Republic}

\ead{drapal@karlin.mff.cuni.cz}

\fntext[D2]{Supported by institutional grant MSM 0021620839. An early version
of this paper was written during Fulbright research stay at the University of
Wisconsin-Madison.}

\author[V]{Petr Vojt\v{e}chovsk\'y}

\address[V]{Department of Mathematics, University of Denver,
2360 S Gaylord St, Denver, Colorado 80208, U.S.A.}

\ead{petr@math.du.edu}

\begin{abstract}
We study a generalization of the classical correspondence between homogeneous
quadratic polynomials, quadratic forms, and symmetric/alternating bilinear
forms to forms in $n$ variables. The main tool is combinatorial polarization,
and the approach is applicable even when $n!$ is not invertible in the
underlying field.
\end{abstract}


\end{frontmatter}



\section{Introduction}

Let $F$ be a field of characteristic $\chr{F}$, and let $V$ be a
$d$-dimensional vector space over $F$. Recall that a \emph{quadratic form}
$\alpha:V\to F$ is a mapping such that
\begin{equation}\label{Eq:QuadraticHomogeneous}
    \alpha(au)=a^2\alpha(u)
\end{equation}
for every $a\in F$, $u\in V$, and such that $\varphi:V^2\to F$ defined by
\begin{equation}\label{Eq:QuadraticPolarization}
    \varphi(u,v) = \alpha(u+v)-\alpha(u)-\alpha(v)
\end{equation}
is a symmetric bilinear form.

The name ``quadratic form'' is justified by the fact that quadratic forms $V\to
F$ are in one-to-one correspondence with homogeneous quadratic polynomials over
$F$. This is a coincidence, however, and it deserves a careful look:

Assume that $\chr{F}\ne 2$. Given a symmetric bilinear form $\varphi:V^2\to F$,
the mapping $\alpha:V\to F$ defined by
\begin{equation}\label{Eq:QuadraticRecovery}
    \alpha(u) = \frac{\varphi(u,u)}{2}
\end{equation}
is clearly a quadratic form satisfying (\ref{Eq:QuadraticPolarization}).
Conversely, if $\alpha$ is a quadratic form with associated symmetric bilinear
form $\varphi$ then (\ref{Eq:QuadraticRecovery}) follows, so $\alpha$ can be
recovered from $\varphi$. Quadratic forms $V\to F$ are therefore in one-to-one
correspondence with symmetric bilinear forms $V^2\to F$. Moreover, upon
choosing a basis $\{e_1,\dots,e_d\}$ of $V$, (\ref{Eq:QuadraticRecovery}) can
be rewritten in coordinates as
\begin{displaymath}
    \alpha(\sum_i a_ie_i) = \sum_{i,j} \frac{a_ia_j}{2}\varphi(e_i,e_j),
\end{displaymath}
showing that $\alpha$ is indeed a homogeneous quadratic polynomial. Every
homogeneous quadratic polynomial is obviously a quadratic form.

Now assume that $\chr{F}=2$. For an alternating bilinear form $\varphi:V^2\to
F$, the homogeneous quadratic polynomial
\begin{equation}\label{Eq:UniqueQuadraticRecovery}
    \beta(\sum_i a_ie_i) = \sum_{i<j}a_ia_j\varphi(e_i,e_j)
\end{equation}
satisfies
\begin{multline*}
    \beta(u+v)-\beta(u)-\beta(v) = \sum_{i<j}(a_ib_j+b_ia_j)\varphi(e_i,e_j)\\
    = \varphi(\sum_i a_ie_i,\sum_j b_je_j) = \varphi(u,v),
\end{multline*}
and thus every alternating bilinear form arises in association with some
quadratic form. Conversely, if $\varphi$ is the symmetric bilinear form
associated with the quadratic form $\alpha$, (\ref{Eq:QuadraticPolarization})
implies that $\varphi$ is alternating. Furthermore, with $\beta$ as in
(\ref{Eq:UniqueQuadraticRecovery}), we see that $\gamma = \alpha-\beta$
satisfies $\gamma(u+v)=\gamma(u)+\gamma(v)$. In particular,
\begin{displaymath}
    \gamma(\sum_i a_ie_i) = \sum_i \gamma(a_ie_i) = \sum_i a_i^2\gamma(e_i),
\end{displaymath}
proving that $\alpha$ is a homogeneous quadratic polynomial. Thus we again have
the desired correspondence between quadratic forms and homogeneous quadratic
polynomials. However, the alternating bilinear form $\varphi$ associated with
$\alpha$ does not determine $\alpha$ uniquely.

\vskip 3mm

The goal of this paper is to investigate generalizations of the three concepts
(quadratic form, homogeneous quadratic polynomial and symmetric resp.
alternating bilinear form) for any number $n$ of variables, giving rise to
polynomial $n$-applications, a class of polynomials of combinatorial degree
$\le n$, and characteristic $n$-linear forms, respectively.

The key insight, which goes back at least to Greenberg \cite{Greenberg}, is the
observation that (\ref{Eq:QuadraticPolarization}) is a special case of the
so-called \emph{polarization} of $\alpha$, but many more concepts and
observations, most of them new, will be required.

The difficulties encountered with quadratic forms over fields of characteristic
two will be analogously encountered for forms in $n$ variables over fields in
which $n!$ is not invertible. There are surprises for $n>3$ (not all
$n$-applications are polynomial) and especially for $n>4$ (not all polynomial
$n$-applications are homogeneous of degree $n$).

Finally, we remark that this paper was not written to mindlessly generalize the
concept of a quadratic form. Rather, it grew from our need to understand why
the prime three behaves differently from all other primes in Richardson's odd
code loops \cite{Richardson}. The reason turned out to be the fact that odd
code loops are connected to trilinear forms satisfying $\varphi(u,u,u)=0$. The
details of this connection to code loops, and thus indirectly to the Monster
group, will be presented separately in a later paper.

\section{Polarization, polynomial mappings, and $n$-applications}

In this paper, a \emph{form} is any mapping $V^n\to F$. A form $f:V^n\to F$ is
\emph{symmetric} if $f(v_1,\dots,v_n) = f(v_{\sigma(1)},\dots, v_{\sigma(n)})$
for every $v_1$, $\dots$, $v_n\in V$ and every permutation $\sigma$ of
$\{1,\dots,n\}$. A symmetric form $f:V^n\to F$ is \emph{$n$-additive} if
$f(u+w,v_2,\dots,v_n) = f(u,v_2,\dots,v_n) + f(w,v_2,\dots,v_n)$ for every $u$,
$w$, $v_2$, $\dots$, $v_n\in V$, and it is \emph{$n$-linear} if it is
$n$-additive and $f(av_1,v_2,\dots,v_n) = af(v_1,v_2,\dots,v_n)$ for every
$a\in F$, $v_1$, $\dots$, $v_n\in V$.

\subsection{Polarization}

Let $\alpha:V\to F$ be a form satisfying $\alpha(0)=0$, and let $n\ge 1$. As in
Ward \cite{Ward}, the \emph{$n$th defect} (also called the \emph{$n$th derived
form}) $\df{n}{\alpha}:V^n\to F$ of $\alpha$ is defined by
\begin{equation}\label{Eq:CombPol}
    \df{n}{\alpha}(u_1,\dots,u_n) = \sum_{1\le i_1<\cdots<i_m\le n}
        (-1)^{n-m}\alpha(u_{i_1}+\cdots+u_{i_m}).
\end{equation}
Then $\df{n}{\alpha}$ is clearly a symmetric form, and it is not hard to see,
using the inclusion-exclusion principle, that the defining identity
(\ref{Eq:CombPol}) is equivalent to the recurrence relation
\begin{align}\label{Eq:Recurrence}
    \df{n}{\alpha}(u_1,\dots,u_n) = \quad &\df{n-1}{\alpha}(u_1+u_2,u_3,\dots,u_n)\notag\\
        \quad-&\df{n-1}{\alpha}(u_1,u_3,\dots,u_n)\\
        \quad-&\df{n-1}{\alpha}(u_2,u_3,\dots,u_n).\notag
\end{align}

If there is a positive integer $n$ such that $\df{n}{\alpha}\ne 0$ and
$\df{n+1}{\alpha}=0$, we say that $\alpha$ has \emph{combinatorial degree} $n$,
and we write $\combdeg{\alpha}=n$. If $\alpha$ is the zero map, we set
$\combdeg{\alpha}=-1$.

Whenever we speak of combinatorial polarization or combinatorial degree of a
form $\alpha:V\to F$, we tacitly assume that $\alpha(0)=0$.

It follows from the recurrence relation (\ref{Eq:Recurrence}) that
$\df{m}{\alpha}=0$ for every $m>\combdeg{\alpha}$. The same relation also shows
that $\combdeg{\alpha}=n$ if and only if $\df{n}{\alpha}\ne 0$ is a symmetric
$n$-additive form. In particular, when $F$ is a prime field,
$\combdeg{\alpha}=n$ if and only if $\df{n}{\alpha}\ne 0$ is a symmetric
$n$-linear form.

Note that combinatorial polarization is a linear process, i.e.,
$\df{n}(c\alpha+d\beta) = c\df{n}{\alpha}+d\df{n}{\beta}$ for every $c$, $d\in
F$ and $\alpha$, $\beta:V\to F$.

In the terminology of Ferrero and Micali \cite{FM}, a form $\alpha:V\to F$ is
an \emph{$n$-application} if
\begin{eqnarray}
    &&\alpha(au) = a^n\alpha(u)
    \text{ for every $a\in F$, $u\in V$, and}\label{Eq:nApp1}\\
    &&\df{n}{\alpha}:V^n\to F\text{ is a symmetric $n$-linear form}.\label{Eq:nApp2}
\end{eqnarray}
Note that (\ref{Eq:nApp1}) and (\ref{Eq:nApp2}) are generalizations of
(\ref{Eq:QuadraticHomogeneous}) and (\ref{Eq:QuadraticPolarization}), that is,
quadratic forms are precisely $2$-applications.

\subsection{Polynomial mappings and $n$-applications}

Let $F[x_1,\dots,x_d]$ be the ring of polynomials in variables $x_1$, $\dots$,
$x_d$ with coefficients in $F$. Denote multivariables by
$\mvar{x}=(x_1,\dots,x_d)$, multiexponents by $\mvar{m} = (m_1,\dots,m_d)$, and
write $\mmon{x}{m}$ instead of $x_1^{m_1}\cdots x_d^{m_d}$. Then every
polynomial $f\in F[\mvar{x}]$ can be written uniquely as a finite sum of
monomials
\begin{displaymath}
    f(\mvar{x}) = \sum c(\mvar{m})\mmon{x}{m},
\end{displaymath}
where $c(\mvar{m})\in F$ for every multiexponent $\mvar{m}$. Finally, let $M(f)
= \{\mvar{m};\;c(\mvar{m})\ne 0\}$ be the set of all multiexponents of $f$.

The \emph{degree of} $f\in F[\mvar{x}]$ is $\deg(f) =
\max\{m_1+\cdots+m_d;\;(m_1,\dots,m_d)\in M(f)\}$.

Define a binary relation $\sim$ on $F[\mvar{x}]$ as follows: For a variable
$x_i$ and exponents $m_i$, $n_i$ let $x_i^{m_i}\sim x_i^{n_i}$ if and only if
either $m_i=n_i$, or $m_i>0$, $n_i>0$ and $m_i-n_i$ is a multiple of $|F|-1$.
(When $F$ is infinite, $m_i-n_i$ is a multiple of $|F|-1$ if and only if
$m_i=n_i$.) Then let $c(\mvar{m})\mmon{x}{m}\sim c(\mvar{n})\mmon{x}{n}$ if and
only if $c(\mvar{m}) = c(\mvar{n})$ and $x_i^{m_i}\sim x_i^{n_i}$ for every
$1\le i\le d$. It is not difficult to see that $\sim$ extends linearly into an
equivalence on $F[\mvar{x}]$.

We call $F[\mvar{x}]/{\sim}$ \emph{reduced polynomials}. Given a polynomial
$f\in F[\mvar{x}]$, the equivalence class $[f]_\sim$ contains a unique
polynomial $g$ such that $0\le m_i<|F|$ for every $1\le i\le d$, $\mvar{m}\in
M(g)$. We usually identify $[f]_\sim$ with this representative $g$, and refer
to $g$ as a reduced polynomial, too.

The significance of reduced polynomials rests in the fact that they are
precisely the polynomial functions:

\begin{lemma}\label{Lm:Reduced} Let $f$, $g\in F[\mvar{x}]$. Then $[f]_\sim =
[g]_\sim$ if and only if $f-g$ is the zero function.
\end{lemma}

Let $\alpha:V\to F$ be a mapping and $B=\{e_1,\dots,e_d\}$ a basis of $V$. Then
$\alpha$ is a \emph{polynomial mapping with respect to $B$} if there exists a
polynomial $f\in F[\mvar{x}]$ such that
\begin{displaymath}
    \alpha(\sum_i a_i e_i) = f(a_1,\dots,a_d)
\end{displaymath}
for every $a_1$, $\dots$, $a_d\in F$. We say that $f$ \emph{realizes $\alpha$
with respect to $B$}. By Lemma \ref{Lm:Reduced}, there is a unique reduced
polynomial realizing $\alpha$ with respect to $B$.

A change of basis will result in a different polynomial representative for a
polynomial mapping, but many properties of the representative remain intact.

\begin{lemma}\label{Lm:ChOB2}
Let $\alpha:V\to F$ be realized with respect to a basis $B$ of $V$ by some
reduced polynomial $f\in F[\mvar{x}]$. If $B^*$ is another basis of $V$ then
$\alpha$ is realized by some reduced polynomial $f^*\in F[\mvar{x}]$ with
respect to $B^*$ and $\deg(f)=\deg(f^*).$
\end{lemma}
\begin{proof}
Let $B=\{e_1,\dots,e_d\}$, $B^*=\{e_1^*,\dots,e_d^*\}$, $e_i^* = \sum_j
c_{i,j}e_j$. Then
\begin{multline*}
    \alpha\bigl(\sum_i a_ie_i^*\bigr) = \alpha\bigl(\sum_i a_i\sum_j
    c_{i,j}e_j\bigr)\\
    = \alpha\bigl(\sum_j \bigl(\sum_i a_ic_{i,j}\bigr) e_j\bigr)
    = f\bigl(\sum_i a_ic_{i,1},\dots,\sum_i a_ic_{i,d}\bigr),
\end{multline*}
which is some polynomial $f^*$ in $a_1$, $\dots$, $a_d$.

We clearly have $\deg(f)=\deg(f^*)$ when $e_1^*=c e_1$ for some $c \ne 0$ and
$e_i^*=e_i$ for every $i>1$. We can therefore assume that $e_1^*=e_1+e_2$ and
$e_i^*=e_i$ for every $i>1$. (Every change of basis is a product of these two
types of elementary operations.)

Let $g(\mvar{x}) = \mmon{x}{m}$ be a monomial of $f$ such that
$\deg(g)=\deg(f)$. Then
\begin{equation}\label{Eq:Monom1}
    g\bigl(\sum_i x_ic_{i,1},\dots,\sum_i x_ic_{i,d}\bigr) =
    (x_1+x_2)^{m_1}x_2^{m_2}\cdots x_d^{m_d}
\end{equation}
contains the reduced monomial $g(\mvar{x})$ as a summand that cannot be
cancelled with any other summand of (\ref{Eq:Monom1}), nor any other summand of
$f^*$, due to $\deg(g)=\deg(f)$. This means that $\deg(f^*)\ge
\deg(g)=\deg(f)$, and the other inequality follows by symmetry.
\end{proof}

We say that a mapping $\alpha:V\to F$ is a \emph{polynomial mapping of degree
$n$} if $\alpha$ is realized by a reduced polynomial of degree $n$ with respect
to some (and hence every) basis of $V$.

We have seen in the Introduction that every $2$-application is a polynomial
mapping, in fact a homogeneous quadratic polynomial. It is a fascinating
question whether every $n$-application is a polynomial mapping, and the series
of papers \cite{Pr1}--\cite{Pr5} by Pr\'oszy\'nski is devoted to this question,
albeit in the more general setting of mappings between modules.

Of course, every $n$-application $V\to F$ is a polynomial mapping when $F$ is
finite, since any mapping $V\to F$ is then a polynomial by Lagrange's
Interpolation. Pr\'oszynski proved that any $3$-application is a polynomial
mapping \cite[Theorem 4.4]{Pr1}, and showed after substantial effort that for
every $n>3$ there is an $n$-application over a field of characteristic two that
is not a polynomial mapping \cite[Example 4.5]{Pr4}.

For $n>3$, there is therefore no hope of maintaining the correspondence between
$n$-applications and a certain class of polynomials, unless we restrict our
attention to polynomial $n$-applications.

We present a characterization of polynomials that are $n$-applications in
Section \ref{Sc:PolyApps}. But first we have a look at forms obtained by
polarization.

\section{Characteristic forms}

For all fields $F$ containing the rational numbers, we will find it convenient
to set $\chr{F}=\infty$, rather than the more contemporary $\chr{F}=0$.

Since we will often deal with repeated arguments, we adopt the following
notation from multisets, cf. \cite{Brualdi}: For an integer $r$ and a vector
$u$, we understand by $r*u$ that $u$ is used $r$ times. For instance,
$\varphi(r*u, s*v)$ stands for
\begin{eqnarray*}
    \varphi(&\underbrace{u,\quad\dots,\quad u},&\underbrace{v,\quad\dots,\quad v}\ \ ).\\
    &\textrm{$r$ times}&\quad \textrm{$s$ times}
\end{eqnarray*}

With these conventions in place, a symmetric form $\varphi:V^n\to F$ is said to
be \emph{characteristic} if either $n<\chr{F}$, or $n\ge\chr{F}=p$ and
$\varphi(p*u,v_1,\dots,v_{n-p})=0$ for every $u$, $v_1$, $\dots$, $v_{n-p}\in
V$. Note that every symmetric form in characteristic $\infty$ is
characteristic.

All forms arising by polarization are characteristic:

\begin{lemma}\label{Lm:NotRealized}
Let $\alpha:V\to F$ and $n\ge 1$. Then $\df{n}{\alpha}:V^n\to F$ is a
characteristic form.
\end{lemma}
\begin{proof}
There is nothing to prove when $n<\chr{F}$. Assume that $n\ge p=\chr{F}$ and
let $u$, $v_1$, $\dots$, $v_{n-p}\in V$. By definition of $\df{n}{\alpha}$,
\begin{displaymath}
    \df{n}{\alpha}(p*u,v_1,\dots,v_{n-p})
    = \sum \sum_{k=0}^p (-1)^{n-r-k}
        \binom{p}{k}\alpha(ku+v_{i_1}+\cdots+v_{i_r}),
\end{displaymath}
where the outer summation runs over all subsets $\{i_1,\dots,i_r\}$ of
$\{1,\dots,n-p\}$. Since $p$ divides $\binom{p}{k}$ unless $k=0$ or $k=p$, the
inner sum reduces to
\begin{displaymath}
    (-1)^{n-r}\alpha(v_{i_1}+\cdots+v_{i_r})
    + (-1)^{n-r-p}\alpha(v_{i_1}+\cdots+v_{i_r}).
\end{displaymath}
When $p$ is odd, the two signs $(-1)^{n-r}$ and $(-1)^{n-r-p}$ are opposite to
each other, and the inner sum vanishes. When $p$ is even, the two signs are the
same and the inner sum becomes $2\alpha(v_{i_1}+\cdots+v_{i_r})=0$.
\end{proof}

In the rest of this section we show that: (a) every characteristic $n$-additive
form can be realized by polarization if $n!$ is invertible, and (b) every
characteristic $n$-linear form can be realized by polarization of a homogeneous
polynomial of degree $n$ with all exponents less than $\chr{F}$. For (a), we
generalize (\ref{Eq:QuadraticRecovery}) and set
\begin{displaymath}
    \alpha(u) = \frac{\varphi(n*u)}{n!}.
\end{displaymath}
For (b), we generalize (\ref{Eq:UniqueQuadraticRecovery}), once again having to
resort to coordinates.

Result (a) is mentioned without proof by Greenberg \cite[p. 110]{Greenberg} and
it has been rediscovered and proved by Ferrero and Micali in \cite{FM}. To our
knowledge, (b) is new.

\begin{lemma}\label{Lm:Coincide} Let $\varphi$, $\psi:V^n\to F$ be
characteristic $n$-additive forms such that
\begin{displaymath}
    \varphi(u_1,\dots,u_n)=\psi(u_1,\dots,u_n)
\end{displaymath}
whenever $u_1$, $\dots$, $u_n$ are pairwise distinct vectors of $V$. Then
$\varphi=\psi$.
\end{lemma}
\begin{proof}
Assume that $\varphi(s_1*u_1,\dots,s_m*u_m)\ne \psi(s_1*u_1,\dots,s_m*u_m)$ for
some pairwise distinct vectors $u_1$, $\dots$, $u_m$ and positive integers
$s_1$, $\dots$, $s_m$, where $s_1+\cdots+s_m=n$ and where $m$ is as small as
possible. Note that $u_i\ne 0$ for every $i$ by additivity, and $s_i<\chr{F}$
since both $\varphi$, $\psi$ are characteristic.

Suppose for a while that $u_2=k u_1$ for an integer $0<k<\chr{F}$. Then
\begin{multline*}
    k^{s_2} \varphi(s_1*u_1, s_2*u_1, s_3*u_3, \dots, s_m*u_m)
    = \varphi(s_1*u_1,s_2*u_2,\dots,s_m*u_m)\\ \ne
    \psi(s_1*u_1,s_2*u_2,\dots,s_m*u_m)
    = k^{s_2} \psi(s_1*u_1, s_2*u_1, s_3*u_3, \dots, s_m*u_m)
\end{multline*}
and thus
\begin{displaymath}
    \varphi((s_1+s_2)*u_1, s_3*u_3, \dots, s_m*u_m)\ne
    \psi((s_1+s_2)*u_1, s_3*u_3, \dots, s_m*u_m),
\end{displaymath}
a contradiction with minimality of $m$.

We can therefore assume that for every $i\ne j$ and every $0<k<\chr{F}$ we have
$u_i\ne ku_j$. Then $v_1=u_1$, $v_2=2u_1$, $\dots$, $v_{s_1}=s_1u_1$,
$v_{s_1+1}=u_2$, $\dots$, $v_{s_1+s_2}=s_2u_2$, $\dots$, $v_n=s_mu_m$ are $n$
distinct vectors and
\begin{displaymath}
    \varphi(v_1,\dots,v_n) = \varphi(s_1*u_1,\dots,s_m*u_m) \prod_{i=1}^m s_i!
\end{displaymath}
is not equal to
\begin{displaymath}
    \psi(s_1*u_1,\dots,s_m*u_m) \prod_{i=1}^m s_i! = \psi(v_1,\dots,v_n),
\end{displaymath}
a contradiction.
\end{proof}

\begin{proposition}\label{Pr:RecoverFromf}
Let $n<\chr{F}$, and let $\varphi:V^n\to F$ be a characteristic $n$-additive
form. Then $\alpha:V\to F$ defined by
\begin{displaymath}
    \alpha(u)=\frac{\varphi(n*u)}{n!}
\end{displaymath}
satisfies $\df{n}{\alpha}=\varphi$.
\end{proposition}
\begin{proof}
Both $\df{n}{\alpha}$ and $\varphi$ are characteristic since $n<\chr{F}$. By
Lemma \ref{Lm:Coincide}, it suffices to show that
$\df{n}{\alpha}(u_1,\dots,u_n) = \varphi(u_1,\dots,u_n)$ for every pairwise
distinct vectors $u_1$, $\dots$, $u_n$ of $V$. We have
\begin{eqnarray*}
    \df{n}{\alpha}(u_1,\dots,u_n) &=&
    \sum_{1\le i_1<\dots<i_k\le n}
        (-1)^{n-k}\alpha(u_{i_1}+\cdots+u_{i_k})\\
    &=&\frac{1}{n!}\sum_{1\le i_1<\dots<i_k\le n}
        (-1)^{n-k} \varphi(n*(u_{i_1}+\cdots +u_{i_k})).
\end{eqnarray*}
Let $v_1$, $\dots$, $v_m$ be pairwise distinct vectors of $V$ such that $v_1$,
$\dots$, $v_m\in\{u_1,\dots,u_n\}$, and let $1\le s_i\le n$ be such that
$s_1+\cdots+s_m=n$. We count how many times $\varphi(s_1*v_1,\dots,s_m*v_m)$
appears in $\df{n}{\alpha}(u_1,\dots,u_n)$. It appears precisely in those
summands $\varphi(n*(u_{i_1}+\cdots+u_{i_k}))$ satisfying
$\{v_1,\dots,v_m\}\subseteq \{u_{i_1},\dots,u_{i_k}\}$, and then it appears
\begin{displaymath}
    \binom{n}{s_1,\dots,s_m} = \frac{n!}{s_1!\cdots s_m!}
\end{displaymath}
times; a number that is independent of $k$. For a fixed $\ell$, there are
precisely $\binom{n-m}{\ell}$ subsets $\{u_{i_1},\dots,u_{i_{\ell+m}}\}$
containing $\{v_1,\dots,v_m\}$. Altogether, $\varphi(s_1*v_1,\dots,s_m*v_m)$
appears with multiplicity
\begin{equation}\label{Eq:Count}
    \binom{n}{s_1,\dots,s_m}\sum_{\ell=0}^{n-m}
    (-1)^{n-(\ell+m)}\binom{n-m}{\ell}.
\end{equation}
Recall that
\begin{displaymath}
    \sum_{\ell=0}^n (-1)^\ell\binom{n}{\ell} = \left\{\begin{array}{ll}
        1,&n=0,\\
        0,&n>0.
    \end{array}\right.
\end{displaymath}
Hence (\ref{Eq:Count}) vanishes when $m<n$. When $m=n$, we have
$s_1=\cdots=s_n=1$, and so (\ref{Eq:Count}) is equal to $n!$.
\end{proof}

\begin{theorem}[Realizing characteristic $n$-linear forms by polarization]\label{Th:Realization}
Let $\{e_1,\dots,e_d\}$ be a basis of $V$ and let $\varphi:V^n\to F$ be a
characteristic $n$-linear form. Define $\alpha:V\to F$ by
\begin{equation}\label{Eq:Realization}
    \alpha(\sum a_ie_i) =
    \sum_{\substack{t_1+\cdots+t_d=n\\0\le t_i <\chr{F}}}
        \frac{a_1^{t_1}\cdots a_d^{t_d}}{t_1!\cdots t_d!}
        \varphi(t_1*e_1,\dots,t_d*e_d).
\end{equation}
Then $\df{n}{\alpha} = \varphi$. Moreover, $\alpha$ is a homogeneous polynomial
of degree $n$ with all exponents less than $\chr{F}$.
\end{theorem}
\begin{proof}
Let $p=\chr{F}\le\infty$. By $n$-linearity and symmetry of $\varphi$, we have
\begin{displaymath}
    \varphi\left(n*\sum_{i=1}^d a_ie_i\right) =
    \sum_{\substack{t_1+\cdots+t_d=n\\0\le t_i\le n}}
        \binom{n}{t_1,\dots,t_d}a_1^{t_1}\cdots a_d^{t_d}
        \varphi(t_1*e_1,\dots,t_d*e_d).
\end{displaymath}
Since $\varphi$ is characteristic, we can rewrite this as
\begin{equation}\label{Eq:AuxA}
    \varphi\left(n{*}\sum_{i=1}^d a_ie_i\right){=}
    \sum_{\substack{t_1+\cdots+t_d=n\\0\le t_i<p}}
        \binom{n}{t_1,\dots,t_d}a_1^{t_1}\cdots a_d^{t_d}
        \varphi(t_1{*}e_1,\dots,t_d{*}e_d).
\end{equation}
If $n<p$, we can divide (\ref{Eq:AuxA}) by $n!$ and apply Proposition
\ref{Pr:RecoverFromf}. For the rest of the proof assume that $n\ge p$.

Then all summands of the right hand side of (\ref{Eq:AuxA}) vanish, since the
multinomial coefficients $\binom{n}{t_1,\dots,t_d}$ are equal to zero (as
$t_i<p$). In fact, the multiplicity of $p$ in the prime factorization of
$\binom{n}{t_1,\dots,t_d}$, say $p^m$, is the same as the multiplicity of $p$
in the prime factorization of $n!$. Thus, upon formally dividing
(\ref{Eq:AuxA}) by $n!$, the left hand side of (\ref{Eq:AuxA}) becomes
$\varphi(n*u)/n!$ and the right hand side of (\ref{Eq:AuxA}) becomes
$\alpha(u)$. The calculation in the proof of Proposition \ref{Pr:RecoverFromf}
therefore still applies, proving $\df{n}{\alpha}=\varphi$.

Finally, $\alpha$ is obviously a homogeneous polynomial of degree $n$ with all
exponents less than $\chr{F}$.
\end{proof}

\begin{example}[$n=p=3$]
Let $\varphi:V^3\to\gf{3}$ be a characteristic trilinear form. Let
$\{e_1,e_2,e_3\}$ be a basis of $V$, and $u=a_1e_1+a_2e_2+a_3e_3$. Then
\begin{displaymath}
    \varphi(u,u,u) =
    \sum_{i\ne j} 3a_i^2a_j \varphi(e_i,e_i,e_j) +
    \sum_{i<j<k} 6a_ia_ja_k \varphi(e_i,e_j,e_k).
\end{displaymath}
Upon formally dividing this equality by $3!$, we obtain the homogeneous
polynomial from Theorem \ref{Th:Realization}, namely
\begin{displaymath}
    \alpha(u) = \sum_{i\ne j} \frac{a_i^2a_j}{2} \varphi(e_i,e_i,e_j)+
        \sum_{i<j<k} a_ia_ja_k \varphi(e_i,e_j,e_k).
\end{displaymath}
\end{example}

A careful reader might wonder if the property that every exponent is less than
$\chr{F}$ is invariant under a change of basis. In general the answer is
``no'', but for mappings of the form (\ref{Eq:Realization}) the answer is
``yes'', see Lemma \ref{Lm:WellDefined}.

\section{Combinatorial degree of polynomial mappings}

We now wish to return to the question: \emph{Which polynomial mappings are
$n$-applications?} Our task is therefore to characterize polynomial mappings
$\alpha$ that satisfy the homogeneity condition $\alpha(au) = a^n\alpha(u)$ and
for which $\df{n}{\alpha}$ is $n$-linear. When $F$ is a prime field,
$\df{n}{\alpha}$ is $n$-linear if and only if $\df{n}{\alpha}$ is $n$-additive,
which happens if and only if $\combdeg{\alpha}\le n$. We therefore need to know
how to calculate the combinatorial degree of polynomial mappings, which is what
we are going to explain in this section. In the next section, we tackle the
homogeneity condition and the linearity of $\df{n}{\alpha}$ with respect to
scalar multiplication.

Let $t$ be a nonnegative integer and $p$ a prime, where we also allow
$p=\infty$. Then there are uniquely determined integers $t_i$, the
\emph{$p$-adic digits of $t$}, satisfying $0\le t_i<p$ and $t=t_0p^0 + t_1p^1 +
t_2p^2 + \cdots$. In particular, when $p=\infty$, then $t_0=t$ and $t_i=0$ for
$i>0$, using the convention $\infty^0=1$. The \emph{$p$-weight $\weight_p(t)$
of $t$} is the sum $t_0+t_1+t_2+\cdots$.

Let $p=\chr{F}$. The \emph{$p$-degree of a monomial} $\mmon{x}{m}\in
F[\mvar{x}]$ is
\begin{displaymath}
    \deg_p(\mmon{x}{m}) = \sum_{i=1}^d \weight_p(m_i),
\end{displaymath}
and the \emph{$p$-degree of a polynomial} $f\in F[\mvar{x}]$ is
\begin{displaymath}
    \deg_p(f) = \max\{\deg_p(\mmon{x}{m});\; \mvar{m}\in M(f)\}.
\end{displaymath}
In particular, when $p=\infty$, $\deg_p(f) = \deg(f)$.

In \cite{Ward}, Ward showed:

\begin{proposition}\label{Pr:DegreesCoincide}
Let $F$ be a prime field or a field of characteristic $\infty$, $V$ a vector
space over $F$, and $\alpha:V\to F$ a polynomial mapping satisfying
$\alpha(0)=0$. Then $\combdeg{\alpha} = \deg(\alpha)$.
\end{proposition}

He also mentioned \cite[p.\ 195]{Ward} that ``It is not difficult to show that,
in general, the combinatorial degree of a [reduced] nonzero polynomial over
$\gf{q}$, $q$ a power of the prime $p$, is the largest value of the sum of the
$p$-weights of the exponents for the monomials appearing in the polynomial.'' A
proof of this assertion can be found already in \cite{VojtCombTexas}. Here we
prove a more general result for polynomials over any field, not just for
polynomials over finite fields $\gf{q}$. We follow the technique of
\cite{VojtCombTexas} very closely.

When $\mvar{x_1} = (x_{1,1},\dots,x_{1,d})$,
$\mvar{x_2}=(x_{2,1},\dots,x_{2,d})$ are two multivariables, we write
$\mvar{x_1}+\mvar{x_2}$ for the multivariable $(x_{1,1}+x_{2,1}$, $\dots$,
$x_{1,d}+x_{2,d})$. Moreover, when $\mvar{m}=(m_1,\dots,m_d)$ is a
multiexponent, we write $(\mvar{x_1}+\mvar{x_2})^{\mvar{m}}$ for
$(x_{1,1}+x_{2,1})^{m_1}\cdots(x_{1,d}+x_{2,d})^{m_d}$. For $f\in F[\mvar{x}]$
satisfying $f(0)=0$ and for $n\ge 1$ let $\df{n}{f}\in
F[x_{1,1},\dots,x_{1,d},\dots,x_{n,1},\dots,x_{n,d}]$ be defined by
\begin{equation}\label{Eq:PolForPol}
    \df{n}{f}(\mvar{x_1},\dots,\mvar{x_n}) =
    \sum_{1\le i_1<\cdots<i_m\le n}
    (-1)^{n-m} f(\mvar{x_{i_1}}+\cdots+\mvar{x_{i_m}}).
\end{equation}
The \emph{(formal) combinatorial degree $\combdeg{f}$ of $f\in F[\mvar{x}]$} is
the least integer $n$ such that $\df{n}{f}$ is a nonzero polynomial and
$\df{n+1}{f}$ is the zero polynomial, letting again $\combdeg{0}=-1$.

Whenever we speak of combinatorial polarization or combinatorial degree of a
polynomial $f$, we tacitly assume that $f(0)=0$.

We shall show in Theorem \ref{Th:CombDeg} that $\combdeg{f}=\deg_p(f)$ for
every $f\in F[\mvar{x}]$ and in Corollary \ref{Cr:CombDeg} that
$\combdeg{\alpha}=\combdeg{f}$ whenever $\alpha:V\to F$ is a polynomial mapping
realized by $f$ with respect to some basis of $V$.

\begin{lemma}\label{Lm:Disjoint}
If $f$, $g\in F[\mvar{x}]$ satisfy $M(f)\cap M(g)=\emptyset$ then
$M(\df{n}{f})\cap M(\df{n}{g})=\emptyset$ for every $n\ge 1$.
\end{lemma}
\begin{proof}
It suffices to establish the lemma when $f$, $g$ are monomials, since
combinatorial polarization is a linear process. Let $f(\mvar{x}) =
\mmon{x}{m}$. Consider one of the summands $f(\mvar{x_1}+\cdots + \mvar{x_s})$
of $\df{n}{f}(\mvar{x_1},\dots,\mvar{x_n})$, as displayed in
(\ref{Eq:PolForPol}). We have
\begin{displaymath}
    f(\mvar{x_1}+\cdots +\mvar{x_s}) =
    (\mvar{x_1}+\cdots+\mvar{x_s})^{\mvar{m}} =
    (x_{1,1}+\cdots+ x_{s,1})^{m_1}\cdots (x_{1,d}+\cdots+x_{s,d})^{m_d}.
\end{displaymath}
In turn, let $h$ be a summand of $f(\mvar{x_1}+\cdots + \mvar{x_s})$. By the
multinomial theorem, for every $1\le i\le d$, the variables $x_{1,i}$, $\dots$,
$x_{s,i}$ appear in $h$ precisely $m_i$ times, counting multiplicities. Hence
the multiexponent $\mvar{m}$ can be reconstructed from any monomial of
$\df{n}{f}(\mvar{x_1},\dots,\mvar{x_n})$.
\end{proof}

\begin{corollary}\label{Cr:Disjoint}
Assume that $f\in F[\mvar{x}]$ satisfies $f(0)=0$. Then $\combdeg{f} =
\max\{\combdeg{\mmon{x}{m}};\;\mvar{m}\in M(f)\}$.
\end{corollary}

We proceed to determine the combinatorial degree of monomials.

Let $\mvar{m}$, $\mvar{n}$ be two multiexponents. We write $\mvar{m}\le
\mvar{n}$ if $m_i\le n_i$ for every $1\le i\le d$. When $\mvar{m}\le\mvar{n}$,
$\mvar{n}-\mvar{m}$ stands for the multiexponent $(n_1-m_1$, $\dots$,
$n_d-m_d)$. We also let
\begin{displaymath}
    \binom{\mvar{m}}{\mvar{n}} = \prod_{i=1}^d \binom{m_i}{n_i}
    = \prod_{i=1}^d \frac{m_i!}{n_i!(m_i-n_i)!},
\end{displaymath}
with the usual convention $0!=1$.

The following lemma gives a critical insight into defects of monomials.

\begin{lemma}\label{Lm:Crucial}
Let $f(\mvar{x}) = \mmon{x}{m}\in F[\mvar{x}]$. Let $\mvar{x_1}$, $\dots$,
$\mvar{x_s}$ be multivariables. Then
\begin{equation}\label{Eq:Monomial}
    \df{s}{f}(\mvar{x_1},\dots,\mvar{x_s}) =
    \sum
    \binom{\mvar{m_1}}{\mvar{m_2}}\cdots
    \binom{\mvar{m_{s-1}}}{\mvar{m_s}}
    \mmon{x_1}{m_s}\mvar{x_2}^{\,\mvar{m_{s-1}}-\mvar{m_s}}\cdots
    \mvar{x_s}^{\,\mvar{m_1}-\mvar{m_2}},
\end{equation}
where the summation ranges over all chains of multiexponents
$\mvar{m}=\mvar{m_1}>\mvar{m_2}>\cdots >\mvar{m_s}>\mvar{0}$.
\end{lemma}
\begin{proof}
Straightforward calculation shows that
\begin{displaymath}
    (\mvar{x}+\mvar{y})^{\mvar{m}} = \prod_{i=1}^d(x_i+y_i)^{m_i}=
    \prod_{i=1}^d \sum_{n_i=0}^{m_i} \binom{m_i}{n_i} x_i^{n_i} y_i^{m_i-n_i}
\end{displaymath}
is equal to
\begin{displaymath}
    \sum_{0\le \mvar{n}\le \mvar{m}} \prod_{i=1}^d\binom{m_i}{n_i} x_1^{n_1}\cdots
    x_d^{n_d}y_1^{m_1-n_1}\cdots y_d^{m_d-n_d}
    = \sum_{0\le \mvar{n}\le \mvar{m}}
    \binom{\mvar{m}}{\mvar{n}}\mmon{x}{n}\mvar{y}^{\,\mvar{m}-\mvar{n}}.
\end{displaymath}
Since $\df{2}{f}(\mvar{x},\mvar{y}) = (\mvar{x}+\mvar{y})^{\mvar{m}} -
\mmon{x}{m} -\mmon{y}{m}$, the lemma follows for $s=2$.

Assume that the lemma holds for $s\ge 2$. Let $g_{\mvar{m_s}}(\mvar{x}) =
\mvar{x}^{\,\mvar{m_s}}$ and note that we have just proved
\begin{equation}\label{Eq:JustDid}
    \df{2}{g_{\mvar{m_s}}}(\mvar{x_1},\mvar{x_2}) = \sum_{0<\mvar{m_{s+1}}<\mvar{m_s}}
    \binom{\mvar{m_s}}{\mvar{m_{s+1}}}\mmon{x_1}{m_{s+1}}
    \mvar{x_2}^{\,\mvar{m_s}-\mvar{m_{s+1}}}.
\end{equation}
Using an analogy of (\ref{Eq:Recurrence}) for formal polynomials and the
induction assumption, we have
\begin{align*}
    &\df{s+1}{f}(\mvar{x_1},\dots,\mvar{x_{s+1}})\\
    &=\df{s}{f}(\mvar{x_1}+\mvar{x_2},\mvar{x_3},\dots,\mvar{x_{s+1}}) -
    \df{s}{f}(\mvar{x_1},\mvar{x_3},\dots,\mvar{x_{s+1}})
        - \df{s}{f}(\mvar{x_2},\mvar{x_3},\dots,\mvar{x_{s+1}})\\
    &=\sum
        \binom{\mvar{m_1}}{\mvar{m_2}}
        {\cdots}
        \binom{\mvar{m_{s-1}}}{\mvar{m_s}}
        ((\mvar{x_1}{+}\mvar{x_2})^{\,\mvar{m_s}}{-}\mvar{x_1}^{\,\mvar{m_s}}{-}
        \mvar{x_2}^{\,\mvar{m_s}} )
        \mvar{x_3}^{\,\mvar{m_{s-1}}-\mvar{m_s}}\cdots
        \mvar{x_{s+1}}^{\,\mvar{m_1}-\mvar{m_2}}\\
    &=\sum
        \binom{\mvar{m_1}}{\mvar{m_2}}
        {\cdots}
        \binom{\mvar{m_{s-1}}}{\mvar{m_s}}
        \df{2}{g_{\mvar{m_s}}}(\mvar{x_1},\mvar{x_2})
        \mvar{x_3}^{\,\mvar{m_{s-1}}-\mvar{m_s}}\cdots
        \mvar{x_{s+1}}^{\,\mvar{m_1}-\mvar{m_2}},
\end{align*}
where the summation ranges over all chains of multiexponents
$\mvar{m}=\mvar{m_1}>\mvar{m_2}>\cdots
>\mvar{m_s}>\mvar{0}$. We are done upon substituting (\ref{Eq:JustDid}) into
the last equation.
\end{proof}

Note that the multiexponents of $\mvar{x_1}$, $\mvar{x_2}$, $\cdots$,
$\mvar{x_s}$ in the sum of (\ref{Eq:Monomial}) are different for every chain
$\mvar{m}=\mvar{m_1}>\mvar{m_2}>\cdots >\mvar{m_s}>\mvar{0}$. Therefore, by
Lemma \ref{Lm:Disjoint}, the combinatorial degree of $\mmon{x}{m}$ is the
length $s$ of a longest chain $\mvar{m}=\mvar{m_1}>\mvar{m_2}>\cdots
>\mvar{m_s}>\mvar{0}$ satisfying
\begin{equation}\label{Eq:Regularity}
    \binom{\mvar{m_i}}{\mvar{m_{i+1}}}\ne 0
\end{equation}
for every $1\le i<s$, where the inequality is understood in $F$.

Let us call a chain $\mvar{m}=\mvar{m_1}>\mvar{m_2}>\cdots
>\mvar{m_s}>\mvar{0}$ of multiexponents satisfying (\ref{Eq:Regularity})
\emph{regular}.

\begin{lemma}\label{Lm:OneDimChains} Let $n=\sum_{i=0}^\infty n_ip^i$, where
$0\le n_i<p$ for every $i$. Then the length of a longest regular chain for
$\mvar{m}=(n)$ is $\omega_p(n)$.
\end{lemma}
\begin{proof}
There is nothing to prove in characteristic $p=\infty$. Assume that $p<\infty$,
and let $a=\sum_{i=0}^\infty a_ip^i$, $b=\sum_{i=0}^\infty b_ip^i$ be two
integers with $0\le a_i$, $b_i<p$ for every $i$. By Lucas Theorem \cite{Fine},
\begin{displaymath}
    \binom{a}{b} \equiv \prod_{i=0}^\infty \binom{a_i}{b_i}
\end{displaymath}
modulo $p$. Consequently, if $\binom{a}{b}\not\equiv 0$, we must have $a_i\ge
b_i$ for every $i$ since $\binom{a_i}{b_i}$ is not divisible by $p$.

Hence the length $t$ of a longest regular chain for $n$ cannot exceed
$\weight_p(n) = \sum_{i=0}^\infty n_i$. On the other hand, $t\ge \weight_p(n)$
holds, because we can construct a regular chain for $n$ of length
$\weight_p(n)$ by reducing one of the $n_i$s by one in each step.
\end{proof}

\begin{lemma}\label{Lm:Chains}
Let $\mvar{m}=(m_1$, $\dots$, $m_d)$ be a multiexponent. Let
$\mvar{m}=\mvar{m_1}>\mvar{m_2}>\cdots>\mvar{m_s}>\mvar{0}$ be a longest
regular chain for $\mvar{m}$. Then $\mvar{m_i}$, $\mvar{m_{i+1}}$ differ in
exactly one exponent for every $1\le i<s$, and $s=\sum_{i=1}^d \weight_p(m_i)$,
where $p=\chr{F}\le \infty$.
\end{lemma}
\begin{proof}
If $\mvar{m_i}$, $\mvar{m_{i+1}}$ differ in two exponents, we can construct a
longer regular chain by reducing the powers separately. Thus, given the regular
chain $\mvar{m}=\mvar{m_1}>\cdots>\mvar{m_s}>\mvar{0}$, we can construct
another regular chain for $\mvar{m}$ of length $s$, in which we first reduce
only the first exponent, then the second exponent, etc. We are done by Lemma
\ref{Lm:OneDimChains}.
\end{proof}

\begin{example}\label{Ex:Chain} Let us construct a longest regular chain for $(7,4)$ in
characteristic $p=3$. Since $7 = 1\cdot 3^0 + 2\cdot 3^1$ and $4=1\cdot 3^0 +
1\cdot 3^1$, the procedure outlined in the proof of Lemma $\ref{Lm:Chains}$
yields the chain $(7,4)>(4,4)>(1,4)>(0,4)>(0,1)>(0,0)$, for instance. The chain
has length $5=\weight_3(7)+\weight_3(4)$, as expected.
\end{example}

We summarize Corollary \ref{Cr:Disjoint} and Lemmas \ref{Lm:Crucial},
\ref{Lm:Chains}:

\begin{theorem}[Combinatorial degree of formal polynomials]\label{Th:CombDeg}
Let $f\in F[\mvar{x}]$ be a polynomial satisfying $f(0)=0$, and let
$\chr{F}=p\le\infty$. Then $\combdeg{f} = \deg_p(f)$.
\end{theorem}

We now return to combinatorial polarization of polynomial mappings. First
observe:

\begin{lemma}\label{Lm:Reducedfn}
Let $f\in F[\mvar{x}]$ be a reduced polynomial satisfying $f(0)=0$. Then
$\df{n}{f}\in F[\mvar{x_1},\dots,\mvar{x_n}]$ is a reduced polynomial for every
$n\ge 1$.
\end{lemma}

\begin{lemma}
Let $\alpha:V\to F$ be a polynomial mapping satisfying $\alpha(0)=0$, and
assume that the reduced polynomial $f\in F[\mvar{x}]$ represents $\alpha$ with
respect to some basis of $V$. Then $\combdeg{f}=\combdeg{\alpha}$.
\end{lemma}
\begin{proof}
Let $\{e_1$, $\dots$, $e_d\}$ be the underlying basis. Let $n\ge 1$, and $u_i =
\sum_j a_{ij}e_j$. Then
\begin{multline}\label{Eq:FormPol}
    \df{n}{\alpha}(u_1,\dots,u_n) = \alpha\bigl(
        \sum_j a_{1j}e_1,\dots,\sum_j a_{nj}e_j\bigr)\\
    = \df{n}{f}( a_{11},\dots,a_{1d},\dots,a_{n1},\dots,a_{nd} ).
\end{multline}
This equality implies $\combdeg{f}\ge \combdeg{\alpha}$, since if
$\df{n}{\alpha}\ne 0$ then $\df{n}{f}$ is a nonzero function and thus a nonzero
polynomial.

On the other hand, assume that $n=\combdeg{f}$. Then $\df{n}{f}$ is a nonzero
polynomial that is reduced by Lemma \ref{Lm:Reducedfn}. Thus $\df{n}{f}$ is a
nonzero function by Lemma \ref{Lm:Reduced}, and (\ref{Eq:FormPol}) implies that
$\combdeg{\alpha}\ge n = \combdeg{f}$.
\end{proof}

\begin{corollary}[Combinatorial degree of polynomial mappings]\label{Cr:CombDeg}
Let $V$ be a vector space over a field $F$ of characteristic $p\le\infty$, and
let $\alpha:V\to F$ be a nonzero polynomial mapping satisfying $\alpha(0)=0$.
Then $\combdeg{\alpha}$ is equal to $\deg_p(f)$, where $f\in F[x_1,\dots,x_d]$
is a reduced polynomial that realizes $\alpha$ with respect to some basis of
$V$.
\end{corollary}

\section{Polynomial $n$-applications}\label{Sc:PolyApps}

\subsection{Totally reduced polynomials}

We have already established that the degree of a polynomial mapping is
well-defined, cf. Lemma \ref{Lm:ChOB2}. By Corollary \ref{Cr:CombDeg}, the
combinatorial degree is also well-defined for polynomial mappings.

However, one has to be careful with even the most common concepts, such as the
property of being homogeneous. To wit, consider the polynomial mapping
$\alpha:\gf{4}^2\to\gf{4}$ defined by $\alpha(a_1e_1+a_2e_2) = a_1^2a_2^2$ with
respect to some basis $\{e_1,e_2\}$ of $\gf{4}^2$ over $\gf{4}$. Then
\begin{multline*}
    \alpha(a_1(e_1+e_2)+a_2e_2) = \alpha(a_1e_1 + (a_1+a_2)e_2) =\\
    a_1^2(a_1+a_2)^2 = a_1^4 + a_1^2a_2^2 = a_1+a_1^2a_2^2.
\end{multline*}
Thus, as a reduced polynomial, $\alpha$ is homogeneous with respect to
$\{e_1,e_2\}$ but not with respect to $\{e_1+e_2,e_2\}$. Of course, no
difficulties arise with respect to homogeneity if we do not insist that
polynomials be reduced.

Let us consider another property of polynomials familiar to us from Theorem
\ref{Th:Realization}: A polynomial $f\in F[x_1,\dots,x_d]$ is \emph{totally
reduced} if for every $\mvar{m}\in M(f)$ and every $1\le i\le d$ we have $0\le
m_i<\chr{F}$.

Theorem \ref{Th:CombDeg} implies immediately:

\begin{corollary}\label{Cr:TotRed}
Let $f\in F[\mvar{x}]$ be a monomial. Then $\combdeg{f}\le\deg(f)$, and the
equality holds if and only if $f$ is totally reduced.
\end{corollary}

Now, the polynomial mapping $\beta:\gf{4}^2\to\gf{4}$ defined by
$\beta(a_1e_1+a_2e_2) = a_1a_2$ is totally reduced with respect to
$\{e_1,e_2\}$, but
\begin{displaymath}
    \beta(a_1(e_1+e_2)+a_2e_2) = \beta(a_1e_1 + (a_1+a_2)e_2) = a_1(a_1+a_2) =
    a_1^2 + a_1a_2
\end{displaymath}
shows that $\beta$ is not totally reduced with respect to $\{e_1+e_2,e_2\}$.
Hence being totally reduced is not a property of polynomial mappings. But we
have:

\begin{lemma}\label{Lm:WellDefined}
Let $\alpha:V\to F$ be a polynomial mapping satisfying $\alpha(0)=0$ and
realized with respect to the basis $B$ (respectively $B^*$) by the reduced
polynomial $f$ (respectively $f^*$). Assume that every monomial $g$ of $f$
satisfying $\combdeg{g}=\combdeg{f}$ is totally reduced. Then every monomial
$g^*$ of $f^*$ satisfying $\combdeg{g^*}=\combdeg{f^*}$ is totally reduced.
\end{lemma}
\begin{proof}
Let $g$ be a monomial of $f$. Let $h^*$ be the reduced polynomial obtained from
$g$ by the change of basis from $B$ to $B^*$, and let $g^*$ be a summand of
$h^*$. Then $\combdeg{g^*}\le
\combdeg{h^*}=\combdeg{g}\le\combdeg{f}=\combdeg{f^*}$ by Corollary
\ref{Cr:CombDeg}, and $\deg(g^*)\le\deg(h^*)\le\deg(g)$. If
$\combdeg{g^*}<\combdeg{f^*}$, there is nothing to prove. Assume therefore that
$\combdeg{g^*}=\combdeg{f^*}$. Then $\combdeg{g^*}=\combdeg{g}=\combdeg{f}$,
and so $g$ is totally reduced by assumption. By Corollary \ref{Cr:TotRed},
$\deg(g)=\combdeg{g}$. But then $\deg(g^*) \le \deg(g) =\combdeg{g} =
\combdeg{g^*}$, and the same corollary shows that $\deg(g^*)=\combdeg{g^*}$ and
that $g^*$ is totally reduced.
\end{proof}

The reader shall have no difficulty establishing:

\begin{lemma}\label{Lm:WellDefined2}
Let $\alpha:V\to F$ be a polynomial mapping satisfying $\alpha(0)=0$ and
realized with respect to the basis $B$ (respectively $B^*$) by the reduced
polynomial $f$ (respectively $f^*$). Assume that there is an integer $n$ such
that $0\ne\deg(g)\equiv n\pmod {|F|-1} )$ for every monomial $g$ of $f$. Then
$0\ne\deg(g^*)\equiv n\pmod {|F|-1}$ for every monomial $g^*$ of $f^*$.
\end{lemma}

Let $B$ be a basis of $V$, $\alpha:V\to F$ a polynomial mapping, and $f$ the
unique reduced polynomial realizing $\alpha$ with respect to $B$. We say that
$\beta:V\to F$ is a \emph{monomial of $\alpha$} if $\beta$ is a polynomial
mapping realized by a monomial of $f$.

Thanks to Lemmas \ref{Lm:WellDefined} and \ref{Lm:WellDefined2}, we can safely
define the following subspaces of polynomial mappings $V\to F$ without having
to fix a basis of $V$:
\begin{align*}
    &\pl{V}{n} {=} \{\alpha;\;\combdeg{\alpha}\le n,\,\alpha(0)=0\},\\
    &\tpl{V}{n} {=} \{\alpha\in\pl{V}{n};\;\textrm{all monomials $\beta$
        with $\combdeg{\beta}{=}n$ are totally reduced}\},\\
    &\dpl{V}{n} {=} \{\alpha;\; \textrm{all monomials $\beta$ satisfy }
        0\ne\deg(\beta)\equiv n\pmod{|F|-1}\}.
\end{align*}
Note that $\pl{V}{n-1}\subseteq \tpl{V}{n}$.

\subsection{Polynomials satisfying $\alpha(au)=a^n\alpha(u)$}

\begin{proposition}\label{Pr:Condition}
Let $\alpha:V\to F$ be a polynomial mapping, and let $n\ge 1$. Then $\alpha$
satisfies $(\ref{Eq:nApp1})$ if and only if $\alpha\in\dpl{V}{n}$.
\end{proposition}
\begin{proof}
Suppose that $\alpha\in\dpl{V}{n}$. Then we can make $\alpha$ into a not
necessarily reduced homogeneous polynomial of degree $n+s(|F|-1)$ for some $s$,
and so $\alpha(au)= a^{n+s(|F|-1)}\alpha(u) = a^n\alpha(u)$.

Conversely, suppose that (\ref{Eq:nApp1}) holds. Let $B=\{e_1,\dots,e_d\}$ be a
fixed basis of $V$, and let $f$ be the reduced polynomial representing $\alpha$
with respect to $B$. Let $M$ be the set of all monomials of $f$, $M^+=\{g\in
M;\;\deg(g)=n+s(|F|-1)$, $s\ge 0\}$, and $M^- = M\setminus M^+$. If $M^-$ is
empty, we are done. Else let $\var{g}$ denote the set of variables present in a
monomial $g$, and let $X$ be a minimal element of $\{\var{g};\;g\in M^-\}$ with
respect to inclusion. Consider a vector $v=\sum_{x_i\in X} a_i e_i$ for some
$a_i\in F$. Let $g^+_1$, $\dots$, $g^+_r$ be all the monomials $g$ of $M^+$
satisfying $\var{g}\subseteq X$, and let $g^-_1$, $\dots$, $g^-_s$ be all the
monomials $g$ of $M^-$ satisfying $\var{g}\subseteq X$. Note that by the
minimality of $X$, $\var{g^-_i}=X$ for every $1\le i\le s$. Set $g^+ =
g^+_1+\cdots+g^+_r$ and $g^- = g^-_1+\cdots+g^-_s$. Let $t_i$ be the degree of
$g^-_i$. For a polynomial $h$, we write $h(v)$ instead of the formally correct
$h(a_1,\dots,a_d)$. Then
\begin{displaymath}
    \alpha(v) = g^+(v)+ g^-(v),
\end{displaymath}
and
\begin{displaymath}
    \alpha(a v) = a^ng^+(v) + a^{t_1}g^-_1(v)
    + \cdots + a^{t_s}g^-_s(v).
\end{displaymath}
On the other hand,
\begin{displaymath}
    a^n\alpha(v) = a^ng^+(v) + a^ng^-(v).
\end{displaymath}
Hence $\alpha(a v)=a^n\alpha(v)$ holds if and only if
\begin{displaymath}
    a^{t_1}g^-_1(v)+\cdots +a^{t_s}g^-_s(v) =
    a^ng^-(v).
\end{displaymath}
Note that $g^-$ is a reduced nonzero polynomial in variables $x_i\in X$. Hence,
by Lemma \ref{Lm:Reduced}, there exists $v=\sum_{x_i\in X}a_ie_i$ such that
$g^-(v)\ne 0$. Fix this vector $v$, and define a polynomial $h$ in one variable
by
\begin{displaymath}
    h(x) = x^ng^-(v)  - x^{t_1}g^-_1(v) - \cdots - x^{t_s}g^-_s(v).
\end{displaymath}
This polynomial is not necessarily reduced, but since $n-t_i\not\equiv
0\pmod{|F|-1}$ for every $1\le i\le s$ and $g^-(v)\ne 0$, it does not reduce to
a zero polynomial. Hence there is $a\in F$ such that $h(a)\ne 0$. But then
$\alpha(a v)\ne a^n\alpha(v)$ with this particular choice of $a$ and $v$, a
contradiction.
\end{proof}

\subsection{Polynomials with $n$-linear defect}

Let $\alpha:V\to F$ be a polynomial mapping of combinatorial degree $n$. Then
$\df{n}{\alpha}$ is a symmetric $n$-additive form. Under which conditions will
$\df{n}{\alpha}$ be $n$-linear? To answer this question, we start with an
example:

\begin{example}
Let $\alpha:\gf{4}\to\gf{4}$, $a\mapsto a^3$. Then there are two longest
regular chains for the (multi)exponent $3$, namely $3>2>0$ and $3>1>0$.
Accordingly, Lemma $\ref{Lm:Crucial}$ yields
\begin{displaymath}
    \df{2}{\alpha}(x,y) = \binom{3}{1}xy^2 + \binom{3}{2}x^2y.
\end{displaymath}
Then $\df{2}{\alpha}(x,ay) = 3xy^2a^2+3x^2ya$, and $a\df{2}{\alpha}(x,y)
=3xy^2a+3x^2ya$. Hence $\df{2}{\alpha}$ is bilinear if and only if $g(x,y,a) =
3xy^2a^2-3xy^2a = 0$ for every $a\in\gf{4}$. Since $g(x,y,a)$ is a reduced
nonzero polynomial (in variables $x$, $y$, $a$), it is a nonzero function by
Lemma $\ref{Lm:Reduced}$, and thus $\df{2}{\alpha}$ is not bilinear. Why did
this happen? Because not every longest regular chain for $3$ ends in $1$.
\end{example}

To resolve the general case, first deduce from Example \ref{Ex:Chain} and
Lemmas \ref{Lm:OneDimChains}, \ref{Lm:Chains}:

\begin{lemma}\label{Lm:LastLink}
Let $f\in F[\mvar{x}]$, $f(\mvar{x})=\mmon{x}{m}$, $\chr{F}=p$. Given a longest
regular chain for $\mvar{m}$, there is $j\ge 0$ such that the chain ends with a
multiexponent $(0$, $\dots$, $0$, $p^j$, $0$, $\dots$, $0)$. Moreover, $j=0$ in
every longest regular chain for $\mvar{m}$ if and only if $f$ is totally
reduced.
\end{lemma}

\begin{proposition}\label{Pr:Linear}
Let $F$ be a field of characteristic $p\le\infty$, and $\alpha:V\to F$ a
polynomial mapping satisfying $\alpha(0)=0$. Then $\df{n}{\alpha}$ is a
characteristic $n$-linear form if and only if $\alpha\in\tpl{V}{n}$.
\end{proposition}
\begin{proof}
By Lemma \ref{Lm:NotRealized}, every $\df{n}{\alpha}$ is characteristic. To
show the equivalence, it suffices to consider a monomial
$\alpha(\mvar{x})=\mmon{x}{m}$, by Lemma \ref{Lm:Disjoint}. If
$\combdeg{\alpha}<n$ then $\df{n}{\alpha}=0$, and vice versa.

Assume that $\combdeg{\alpha}=n$. Longest regular chains for $\mvar{m}$ satisfy
the conclusion of  Lemma \ref{Lm:LastLink}. Let $a\in F$. By Lemma
\ref{Lm:Crucial}, any longest regular chain with $j=0$ contributes the same
monomial to $\df{n}{\alpha}(a\mvar{x_1},\dots,\mvar{x_n})$ and to
$a\df{n}{\alpha}(\mvar{x_1},\dots,\mvar{x_n})$. On the other hand, every
longest regular chain with $j>0$ contributes to
$\df{n}{\alpha}(a\mvar{x_1},\dots,\mvar{x_n})$ by a monomial containing the
power $a^{p^j}$, while it contributes to
$a\df{n}{\alpha}(\mvar{x_1},\dots,\mvar{x_n})$ by a monomial containing the
power $a^1$. Hence, $\df{n}{\alpha}(\mvar{x_1},\dots,a\mvar{x_n}) -
a\df{n}{\alpha}(\mvar{x_1},\dots,\mvar{x_n})$ is a reduced polynomial that is
nonzero if and only if $\alpha$ is not totally reduced. We are then done by
Lemma \ref{Lm:Reduced}.
\end{proof}

\section{The correspondence}

Denote by $\mathcal A_n(V)$ the vector space of polynomial $n$-applications
$V\to F$ and by $\mathcal C_n(V)$ the vector space of characteristic $n$-linear
forms $V^n\to F$.

\begin{theorem}[Correspondence]
Let $V$ be a vector space over $F$. Then
\begin{equation}\label{Eq:Corr1}
    \mathcal A_n(V) = \tpl{V}{n}\cap \dpl{V}{n}
\end{equation}
and
\begin{equation}\label{Eq:Corr2}
    \mathcal C_n(V) \cong (\tpl{V}{n}\cap \dpl{V}{n})/(\pl{V}{n-1}\cap \dpl{V}{n}).
\end{equation}
\end{theorem}
\begin{proof}
The equality (\ref{Eq:Corr1}) follows from Propositions \ref{Pr:Condition} and
\ref{Pr:Linear}. To prove (\ref{Eq:Corr2}), let $\Psi$ be the restriction of
the polarization operator $\Delta^n$ to $\tpl{V}{n}\cap \dpl{V}{n}$. By
Proposition \ref{Pr:Linear}, the image of $\Psi$ is contained in $\mathcal
C_n(V)$. By Theorem \ref{Th:Realization}, $\Psi$ is onto $\mathcal C_n(V)$.
Clearly, the kernel of $\Psi$ consists of $\tpl{V}{n}\cap \dpl{V}{n} \cap
\pl{V}{n-1} = \pl{V}{n-1}\cap \dpl{V}{n}$.
\end{proof}

\begin{corollary} Let $V$ be a $d$-dimensional vector space over a field $F$
with $\chr{F}=\infty$. Then $\mathcal A_n(V)$ are precisely the homogeneous
polynomials of degree $n$ in $d$ variables over $F$, and $\mathcal A_n(V)\cong
C_n(V)$.
\end{corollary}
\begin{proof}
Since $F$ is infinite, $\dpl{V}{n}$ consists of homogeneous polynomials of
degree $n$. The degree and combinatorial degree of polynomials coincide over
$F$, by Theorem \ref{Th:CombDeg}. Hence $\pl{V}{n-1}\cap\dpl{V}{n}$ is trivial.
As all polynomials over $F$ are totally reduced, we have $\tpl{V}{n} =
\pl{V}{n}$ and $\tpl{V}{n}\cap \dpl{V}{n} = \dpl{V}{n}$.
\end{proof}

It is not true in general that $\mathcal A_n(V) = \tpl{V}{n}\cap \dpl{V}{n}$
consists only of homogeneous polynomials of degree $n$, as was first noticed by
Pr\'oszy\'nski in the setting of $n$-applications (he did not work with
combinatorial degrees).

Consider the form $\alpha:\gf{4}^5\to\gf{4}$ defined by
\begin{equation}\label{Eq:StrangeForm}
    \alpha(x_1,x_2,x_3,x_4,x_5) =  x_1x_2x_3x_4x_5 + x_1^2x_2^2x_3^2x_4^2.
\end{equation}
Then $\combdeg{\alpha}=5$ and $\deg(\alpha)=8$. Moreover, the only monomial
$\beta$ of $\alpha$ satisfying $\combdeg{\beta}=5$ is totally reduced, and the
degree of every monomial of $\alpha$ differs from $5$ by a multiple of $3=4-1$.
Hence $\alpha$ is a $5$-application. It cannot be turned into a homogeneous
polynomial of degree $5$ by any change of basis, by Lemma \ref{Lm:ChOB2}. But
it can be made into a homogeneous polynomial of degree $8$, for instance the
polynomial
\begin{displaymath}
   x_1^4x_2x_3x_4x_5 + x_1^2x_2^2x_3^2x_4^2,
\end{displaymath}
no longer reduced.

It appears to be an interesting problem of number-theoretical flavor to
characterize all pairs $(V,n)$ for which $\tpl{V}{n}\cap \dpl{V}{n}$ does
contain only homogeneous polynomials of degree $n$. It is not our intention to
study this problem in detail here. Nevertheless we have the following result
that shows that something interesting happens during the transition from $n=4$
to $n=5$ (also see Sections $2$ and $3$ of \cite{Pr1}):

\begin{proposition}\label{Pr:NotAllHomo}
Let $|F|=q=p^e$ and let $V$ be a $d$-dimensional vector space over $F$. If
$n<5$ then $\tpl{V}{n}\cap \dpl{V}{n}$ consists of homogeneous polynomials of
degree $n$. If $n\ge \max\{5,q\}$, $e\ge 2$, and $d\ge n$ then $\tpl{V}{n}\cap
\dpl{V}{n}$ does not consist only of homogeneous polynomials of degree $n$.
\end{proposition}
\begin{proof}
Let $2\le n<5$, let $\alpha\in \tpl{V}{n}\cap \dpl{V}{n}$, and let $\beta$ be a
monomial of $\alpha$. We show that $\deg(\beta)\le n$. Assume that
$\deg(\beta)=n+s(q-1)$, $s>0$. When $\combdeg{\beta}=n$ then $\beta$ is totally
reduced since $\alpha\in\tpl{V}{n}$, and hence $\deg(\beta)=n$, a
contradiction. Assume therefore that $m=\combdeg{\beta}<n$.

If $m=1$, $\beta$ is a scalar multiple of $x^{p^i}$, $i<e$, and $p^i=n+s(q-1)$.
Since $s>0$, we have $p^i>q$, a contradiction with $i<e$. If $m=2$, $\beta$ is
a scalar multiple of $x^{p^i+p^j}$ or $x^{p^i}y^{p^j}$ for some $i\le j<e$, and
$p^i+p^j=n+s(q-1)$. Since $s>0$, we have $p^i+p^j>q$, thus $p^j>q/2$, so
$p^j\ge q$, a contradiction with $j<e$. If $m=3$, we have $n=4$, and $\beta$ is
a scalar multiple of $x^{p^i+p^j+p^k}$ or $x^{p^i+p^j}y^{p^k}$ or
$x^{p^i}y^{p^j}z^{p^k}$. We can assume that $i\le j\le k<e$, and
$p^i+p^j+p^k=4+s(q-1)$. Suppose that $s>1$. Then $p^i+p^j+p^k>2q$, thus
$p^k>q/2$, a contradiction with $k<e$. Suppose that $s=1$. Then
$p^i+p^j+p^k=q+3$. Since $p^k>q/3$, we must have $p=2$, else $p^k\ge q$. Then
$q+3$ is odd, $p^i=1$, $p^j+p^k>q$, $p^k>q/2$, a contradiction.

Now assume that $n\ge\max\{5,q\}$, $e\ge 2$, and $d\ge n$. Suppose for a while
that there are $0\le a_0$, $\dots$, $a_{e-1}$ such that
\begin{equation}\label{Eq:Cond}
    n+q-1 = a_0p^0 + \cdots + a_{e-1}p^{e-1},\quad
    a_0+\cdots + a_{e-1}<n.
\end{equation}
Since $d\ge n$, we can fix a basis of $V$ and define $\alpha:V\to F$ by setting
$\alpha(x_1,\dots,x_d)$ equal to
\begin{displaymath}
    x_1\cdots x_n + (x_1\cdots
    x_{a_0})(x_{a_0+1}^p\cdots x_{a_0+a_1}^p)\cdots
    (x_{a_0+\cdots+a_{e-2}+1}^{p^{e-1}}\cdots x_{a_0+\cdots+
    a_{e-1}}^{p^{e-1}}).
\end{displaymath}
Furthermore, $\alpha$ so defined satisfies $\deg(\alpha) = n+q-1$, and
\begin{displaymath}
    \combdeg{\alpha} = \max\{n,a_0+\cdots+a_{e-1}\} = n.
\end{displaymath}
The only monomial of $\alpha$ with combinatorial degree equal to
$\combdeg{\alpha}$ is totally reduced, and hence $\alpha$ is an $n$-application
but not a homogeneous polynomial of degree $n$. It remains to show that
(\ref{Eq:Cond}) can be satisfied.

Let $n=sp^e+r$, where $0\le r<p^e$ and $0<s$. Then $n+q-1 = sp^e + r + p^e - 1
= (sp)(p^{e-1}) + r + (p-1)(1+p+\cdots + p^{e-1})$. Hence it is possible to
write $n+q-1 = a_0p^0+ a_1p^1 + \cdots + a_{e-1}p^{e-1}$ with some $0\le a_i$
satisfying $a_0+\cdots + a_{e-1} \le sp+r+(p-1)e$. A short calculation shows
that $sp+r+(p-1)e\le n$ holds if and only if $e\le sp(1+p+\cdots + p^{e-2})$.
Since $e\ge 2$, we have $e\le p^{e-1}\le p(1+p+\cdots +p^{e-2})\le
sp(1+p+\cdots + p^{e-2})$, and the equality holds if and only if $e=2=p$ and
$s=1$.

Assume that $e=2=p$, $s=1$. Then $n\in\{5$, $6$, $7\}$, and it is easy to check
in each case that (\ref{Eq:Cond}) holds with a suitable choice of $a_0$, $a_1$.
(For $n=5$, we recover (\ref{Eq:StrangeForm}).)
\end{proof}

\section{Acknowledgement}

We thank the anonymous referee for several useful comments concerning the
structure of this paper and for pointing out the reference \cite{Greenberg}.

\end{document}